\documentclass[12pt]{amsart}
\usepackage{graphicx}
\usepackage{pdfpages}

\newtheorem{theorem}{Theorem}[section]
\newtheorem{lemma}[theorem]{Lemma}
\newtheorem{definition}[theorem]{Definition}

\theoremstyle{remark}

\numberwithin{equation}{section}

\begin{document}

\title{A generating set of Reidemeister moves of oriented virtual knots}

\author{Danish Ali}
\address{School of mathematical sciences, Dalian University of Technology, China}
\curraddr{}
\email{danishali@dlut.edu.cn}



\date{May 29, 2025}


\begin{abstract}

In oriented knot theory, verifying a quantity is an invariant involves checking its invariance under all oriented Reidemeister moves, a process that can be intricate and time-consuming. A generating set of oriented moves simplifies this by requiring verification for only a minimal subset from which all other moves can be derived. While generating sets for classical oriented Reidemeister moves are well-established, their virtual counterparts are less explored. In this study, we enumerate the oriented virtual Reidemeister moves, identifying seventeen distinct moves after accounting for redundancies due to rotational and combinatorial symmetries. We prove that a four-element subset serves as a generating set for these moves. This result offers a streamlined approach to verifying invariants of oriented virtual knots and lays the groundwork for future advancements in virtual knot theory, particularly in the study of invariants and their computational properties.
\end{abstract}

\maketitle

\section{Introduction}
Knots are considered equivalent if they can be continuously deformed into one another in three-dimensional space without cutting or gluing. This means that we can manipulate the knot in three-dimensional space, but we can not break or reconnect any of its strands. If two knots are topologically equivalent, they represent the same knot. Knots are often represented using knot diagrams, which are two-dimensional representations of the knot with over-crossings and under-crossings information. A knot invariant is a mathematical property or quantity associated with knots that remains constant for any given knot diagram. This means that if two knots are equivalent, they will share the same value as this invariant. Knot invariants are essential in knot theory for distinguishing between different types of knots and proving whether two knots are the same or different. To prove that a certain quantity is indeed a knot invariant, one must show that this quantity does not change when the knot diagram is manipulated through Reidemeister moves.

In classical knot theory, knots are studied as embeddings of circles in 3-dimensional space $\mathbb{R}^3$ or in the 3-sphere $\mathbb{S}^3$. Virtual knot theory was introduced by Louis H. Kauffman as a generalization of classical knot theory \cite{Kauffman}. Virtual knot theory extends classical knot theory by allowing for diagrams that include virtual crossings, representing interactions that cannot be realized in $\mathbb{R}^3$. These virtual crossings provide a way to study knots as embeddings in thickened surfaces. Reidemeister moves are specific local operations we can perform on the diagram to change its appearance while preserving the knot's equivalence \cite{Reidemeister}. Two classical knot diagrams $D$ and $D'$ are equivalent if and only if $D$ can be transformed into $D'$ using classical Reidemeister moves ($C1-C3$), as shown in Fig.\ref{crmu}. Similarly, the two virtual knot diagrams $V$ and $V'$ are equivalent if and only if $V$ can be transformed into $V'$ using classical Reidemeister moves and virtual Reidemeister moves ($V1-V4$), as shown in Fig.\ref{vrmu}. A combination of classical Reidemeister moves and virtual Reidemeister moves is known as generalized Reidemeister moves \cite{danish}.

\begin{figure}
    \centering
    \includegraphics[width=\textwidth]{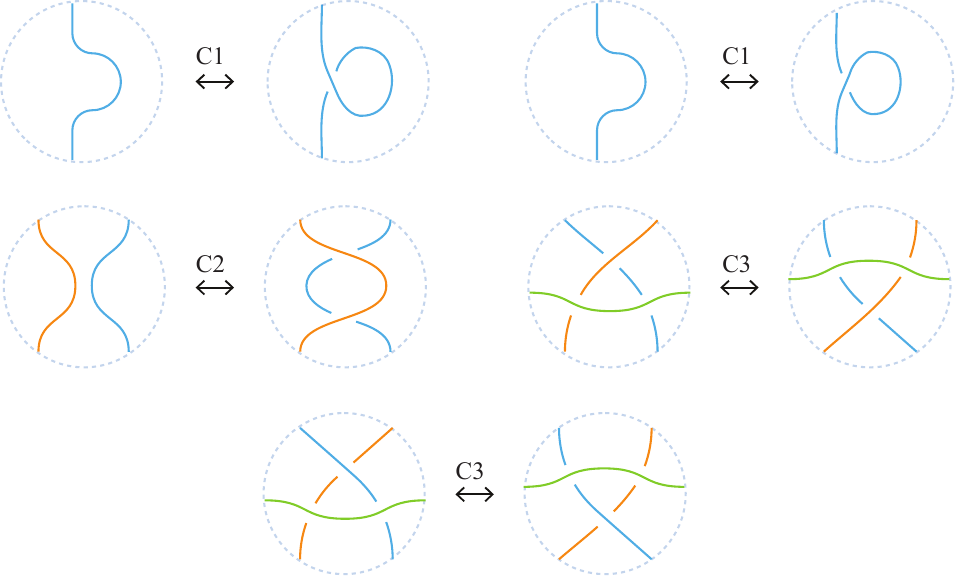}
    \caption{Classical Reidemeister moves}
    \label{crmu}
\end{figure}

\begin{figure}
    \centering
    \includegraphics[width=\textwidth]{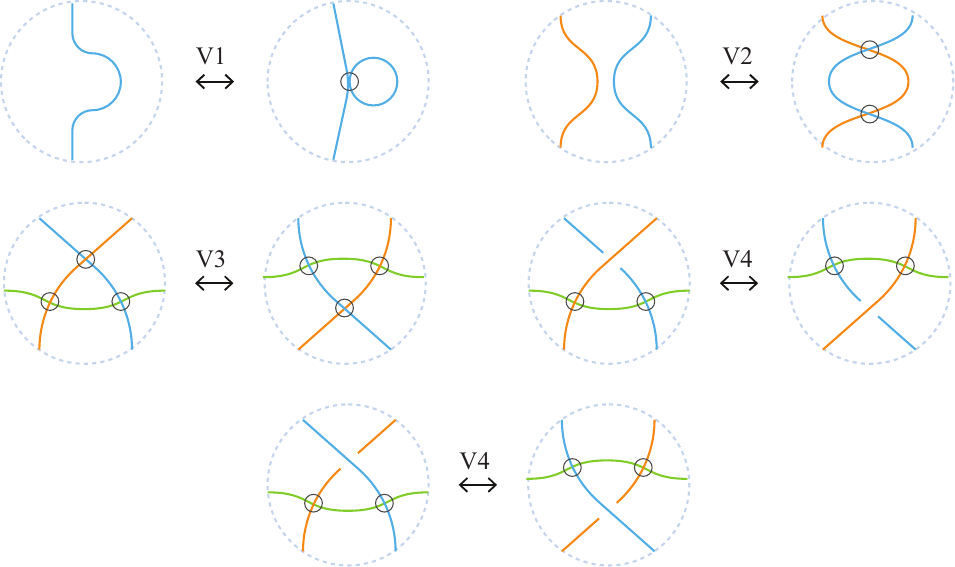}
    \caption{Virtual Reidemeister moves}
    \label{vrmu}
\end{figure}

\begin{figure}
    \centering
    \includegraphics[width=0.9\linewidth]{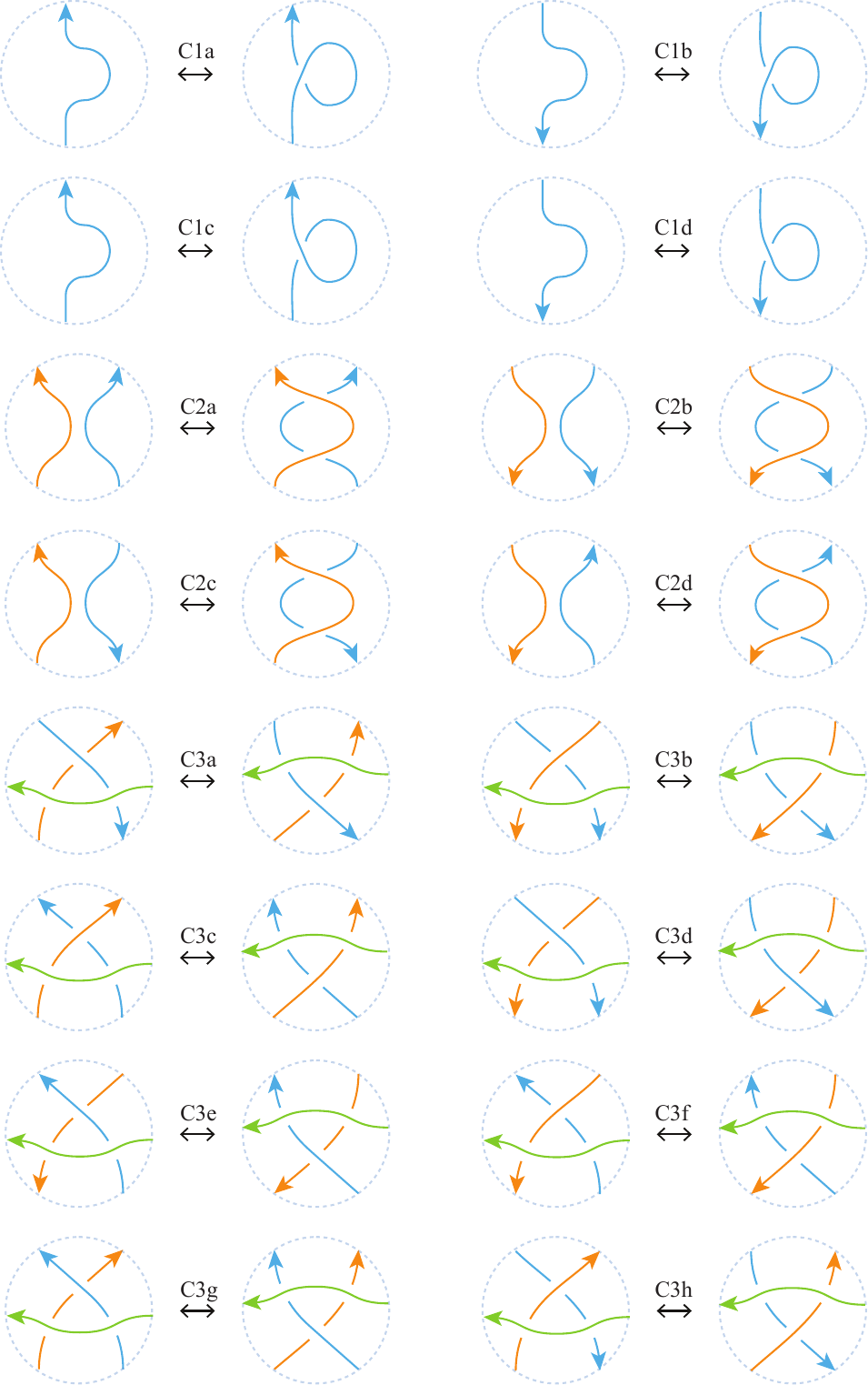}
    \caption{Oriented classical Reidemeister moves}
    \label{crmo}
\end{figure}

\begin{figure}
    \centering
    \includegraphics[width=\linewidth]{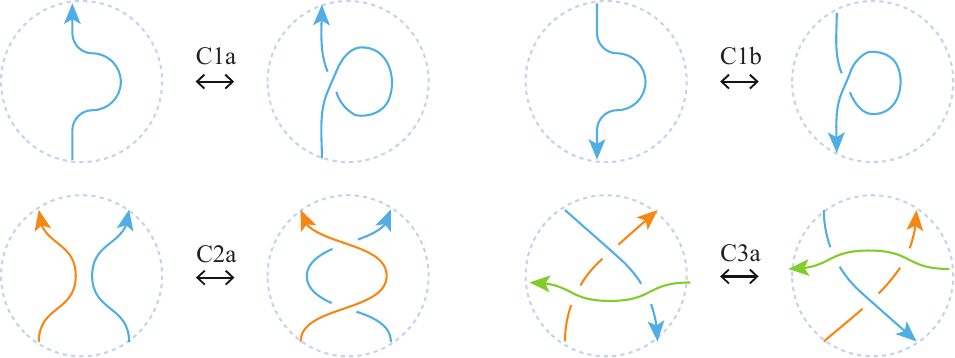}
    \caption{A generating set of classical Reidemeister moves}
    \label{gscrm}
\end{figure}

When working with oriented knots, which have a specified direction (indicated by arrows on the diagram), the versions of the Reidemeister moves need to respect the orientation. This means that when applying these moves, the direction of each strand in the knot diagram must be preserved. Oriented Reidemeister moves are thus crucial for defining knot invariants in the context of oriented knot diagrams, as they ensure that the property being studied is consistent across diagrams that reflect the same knot with a defined direction.

In the study of oriented knot theory, there are 16 oriented versions of classical Reidemeister moves, as shown in Fig.\ref{crmo}. These are variations of the classical Reidemeister moves that take into account the direction (orientation) of the strands in the knot. Verifying that a particular quantity is an invariant for oriented knots requires showing that it remains unchanged under all 16 of these moves. This process can be labour-intensive and complex. To simplify this process, it is advantageous to have a generating set of oriented Reidemeister moves. The main benefit of having a generating set is that it reduces the work required to prove that a certain quantity defined on oriented knot diagrams is indeed an invariant. Instead of checking invariance under all 16 possible moves, one only needs to verify it for the smaller generating set. If the quantity remains invariant under the moves in the generating set, then, by extension, it is invariant under any oriented Reidemeister move due to the fact that these can be constructed from the generating set through finite sequences and isotopes.

\begin{definition}[Generating set]
Let $S$ and $G$ be some sets. The set $G$ is called a generating set of $S$ if every element of $S$ can be represented or expressed in terms of elements from $G$. In other words, $G$ generates $S$ if, for any $e_i \in S$, there exists a way to express each $e_i$ using the elements of $G$ with some operations.
\end{definition}

Polyak introduced the concept of a minimal generating set for classical Reidemeister moves \cite{polyak}. Polyak originally used the symbol $\Omega$ to denote the classical Reidemeister moves, but for clarity and convenience, we use $C$ for classical Reidemeister moves and $V$ for virtual Reidemeister moves. This notation helps to avoid confusion, especially when studying equivalence classes of virtual knots and links. Identifying which moves are classical and which are virtual helps in analyzing the steps involved in proving equivalence or computing invariants for virtual knots. Polyak in \cite{polyak} proved that a minimal generating set of oriented Reidemeister moves for classical knot diagrams consists of four specific oriented Reidemeister moves as shown in Fig.\ref{gscrm}. He demonstrated that all oriented versions of Reidemeister moves can be generated by a set of just four moves and that no fewer than four are sufficient to generate them all.

\begin{theorem}[\cite{polyak}, Theorem 1.1]
Let $D$ and $D'$ be two diagrams in $\mathbb{R}^2$, representing the same oriented link. Then one may pass from $D$ to $D'$ by isotopy and a finite sequence of four oriented Reidemeister moves $C1a, C1b, C2a$ and $C3a$, shown in Fig.\ref{gscrm}.
\end{theorem}

Caprau and Scott in \cite{Caprau} proved that in addition to \{$C1a, C1b, C2a$, $C3a$\} there are eleven more minimal four-element generating sets of oriented Reidemeister moves. They showed that these twelve sets encompass all possible minimal generating sets for oriented classical Reidemeister moves. Singular knots are a generalization of classical knots that allow singular crossings \cite{sheikh}. Analogous work for singular knot theory was conducted by Bataineh et al. in \cite{gsosk}, where they provided a detailed analysis of a generating set of Reidemeister moves for oriented singular link diagrams. Suwara in \cite{Suwara} introduced the concept of directed oriented Reidemeister moves. Specifically, a Reidemeister move of type $C1$ or $C2$ is called forward if it increases the number of crossings in the diagram, and backward if it decreases the number of crossings in the diagram. Suwara proved that the set of 8 directed Polyak moves generates a set of directed oriented Reidemeister moves. The set is proven to be minimal because removing any move would lead to the inability to reconstruct certain transformations needed for equivalence.

Virtual knot theory extends classical knot theory by introducing virtual crossings. This leads to new versions of the Reidemeister moves involving these virtual crossings, requiring more complex transformations to analyze and generate equivalent diagrams. A similar analysis of generating sets for oriented virtual knot diagrams has been less developed. First, we will enumerate all possible oriented virtual Reidemeister moves. To clarify the distinctions between the moves $V1, V2, V3,$ and $V4$, we outline their unique characteristics based on the local configurations and orientations in the oriented link diagram. The $V1$ move involves the addition or removal of a single oriented loop on a strand, the loop can have two orientations (counterclockwise
and clockwise). $V2$ moves concern the creation or elimination of a bigon formed by two strands, differentiated by the relative orientations of the strands (parallel or antiparallel). For $V3$, which involves a triangle configuration with three virtual crossings, distinguished by the orientations of the three strands. Similarly, the $V4$ moves also involved a triangle configuration with one classical crossing and two virtual crossings, distinguished by the orientations of the three strands and the type of classical crossing (positive or negative). 

There are two possible orientation of $V1$ and three possible orientations for the arcs involved in $V2$. $V3$ has only four possible orientations while $V4$ has eight possible orientations as shown in Fig.\ref{vrmo}. A combinatorial approach produces several additional moves, but it becomes clear that some are merely rotations of others. Therefore, these redundant moves are excluded from consideration. Our focus will be only on these seventeen oriented versions of Reidemeister moves, as shown in Fig.\ref{vrmo}.

\begin{figure}
    \centering
    \includegraphics[width=0.8\linewidth]{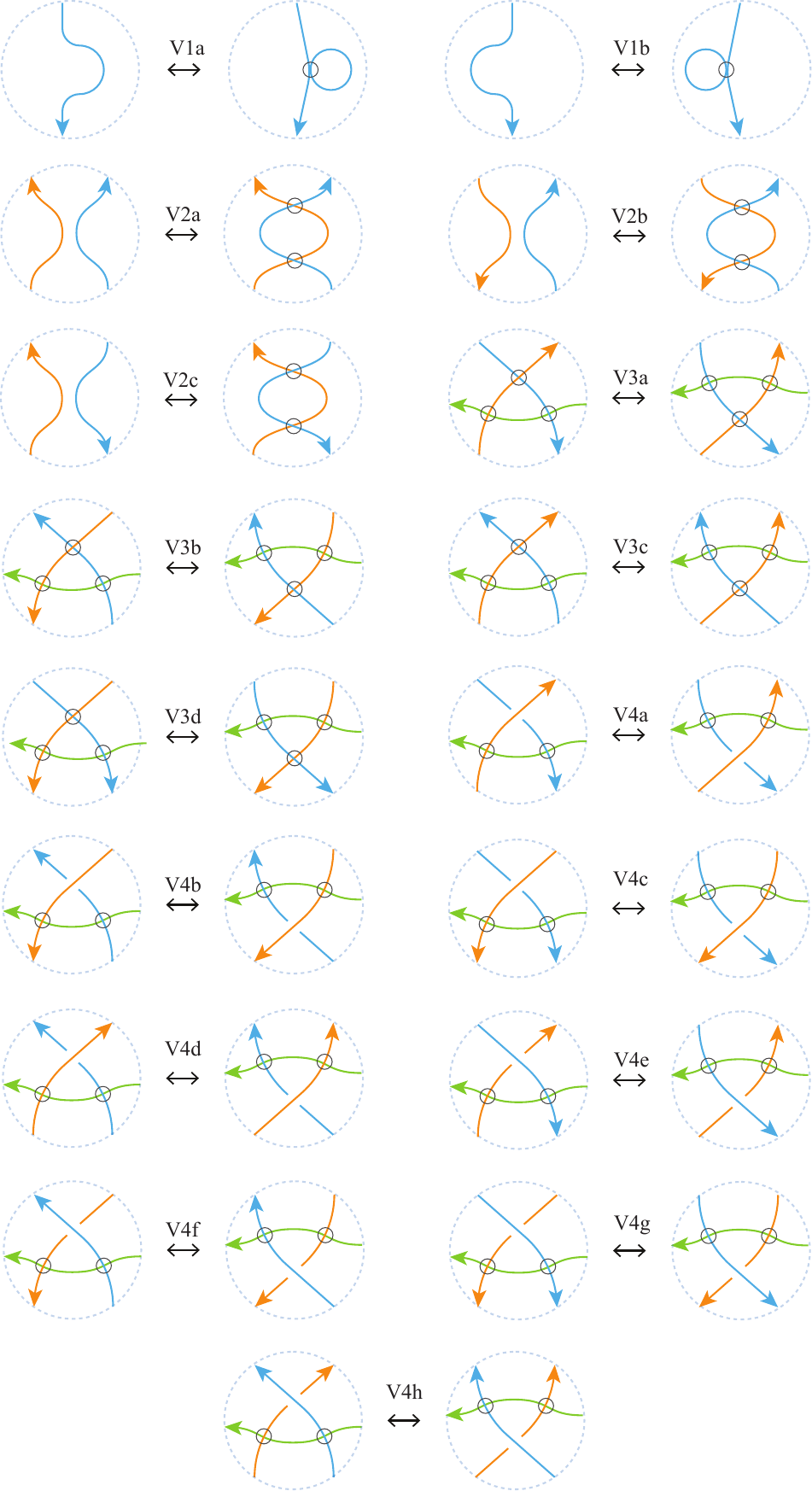}
    \caption{Oriented virtual Reidemeister moves}
    \label{vrmo}
\end{figure}

\begin{figure}[htb]
    \centering
    \includegraphics[width=1\linewidth]{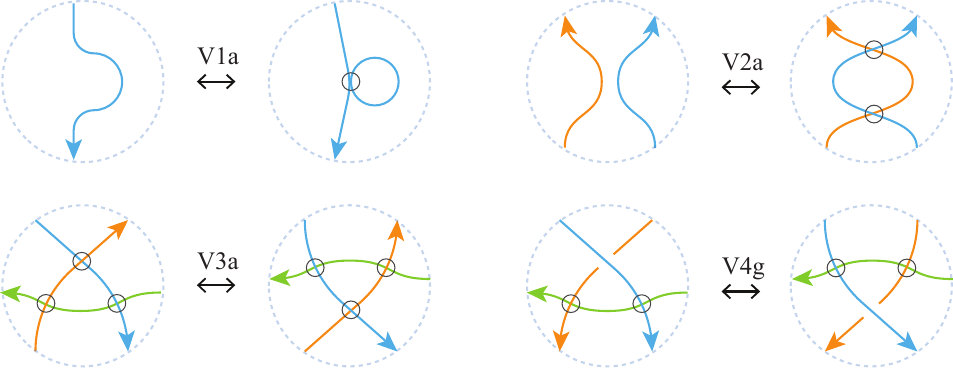}
    \caption{A generating set of virtual Reidemeister  moves}
    \label{gsvrm}
\end{figure}

Understanding a minimal generating set for oriented virtual knot diagrams is crucial. With seventeen oriented versions of Reidemeister moves, it can become tedious to check if a certain quantity defined on oriented virtual knot diagrams yields an invariant for oriented virtual knots. Hence, it is useful to have a generating set of oriented virtual Reidemeister moves to minimize the work required to check for invariance. This paper aims to address the gap in the literature by identifying a generating set of oriented Reidemeister moves for virtual links. We showed that the set of oriented virtual Reidemeister moves shown in Fig.\ref{gsvrm} is a generating set of all possible oriented virtual Reidemeister moves. This exploration will contribute to the foundational work necessary for more advanced research in virtual knot theory and its applications \cite{danishgc}. Interestingly, our
generating set has only one oriented move from each $V1 - V4$, unlike
classical knot theory, where we have two moves for the first oriented
Reidemeister move ($C1$).

\begin{theorem}\label{theorem}
Let $V$ and $V'$ be two oriented virtual knot diagrams in $\mathbb{R}^2$. Then the diagram $V$ can be transformed into $V'$ by isotopy, and a finite sequence of oriented Reidemeister moves $C1a, C1b, C2a, C3a,$ $ V1a, V2a, V3a$ and $V4g$.
\end{theorem}

\section{A generating set of oriented virtual knot diagrams}\label{secproof}

Let $S$ be a set of oriented virtual Reidemeister moves, as shown in Fig.\ref{vrmo}, and $G=\{V1a, V2a, V3a, V4g\}$, then we call $G$ a generating set of $S$ if and only if every move in $S$ can be expressed as a composition (finite sequence) of the moves in $G$. In other words, $G$ should be sufficient to generate or simulate any move in $S$. To establish that $G$ is indeed a generating set, we need to demonstrate that each move in $S$ can be constructed using some combination of the elements in $G$.

This section is dedicated to the proof of Theorem \ref{theorem}. First, we will divide our set $S$ into subsets $S_1, S_2, S_3, S_4$ and $G$ such that $S_1=\{V1b\}$, $S_2=\{V2b, V2c\}$, $S_3=\{V3b, V3c, V3d\}$, $S_4=\{V4a, V4b, V4c, V4d,$ $V4e, V4f, V4h\}$ and $G=\{V1a, V2a, V3a, V4g\}$. Note that $S=G \cup S_1 \cup S_2 \cup S_3 \cup S_4 =$ \{All possible oriented virtual Reidemeister moves, shown in Fig.\ref{vrmo} \}. First, we will show that the moves in $S_1$ and $S_2$ can be generated by the moves in $G$ (Lemmas \ref{Lemma1} to \ref{Lemma1new}). Then we can use the moves in $G \cup S_1 \cup S_2$ to prove that the moves in the set $S_3$ can be generated by a finite sequence of the moves in $G \cup S_1 \cup S_2$ (Lemma \ref{Lemma2}). Finally, we will show that the moves in the set $S_4$ can be generated by the moves in $G \cup S_1 \cup S_2 \cup S_3$ (Lemmas \ref{Lemma3} to \ref{Lemma9}). The proof of Theorem \ref{theorem} will open a series of questions. One of the important questions is to find all possible minimal generating sets of oriented virtual Reidemeister moves.

Our generating set $G$ contains $V1a, V2a, V3a$, and $ V4g$ moves. Therefore, we can show that the move $V2b$ can be realized by a sequence of moves $V1a, V2a$, and $V3a$ moves. 

\begin{lemma}\label{Lemma1}
The move $V2b$ can be realized by a sequence of $V1a, V2a$ and $V3a$ moves.
\end{lemma}

\begin{proof} 
$$\centerline{\includegraphics[width=0.8\textwidth]{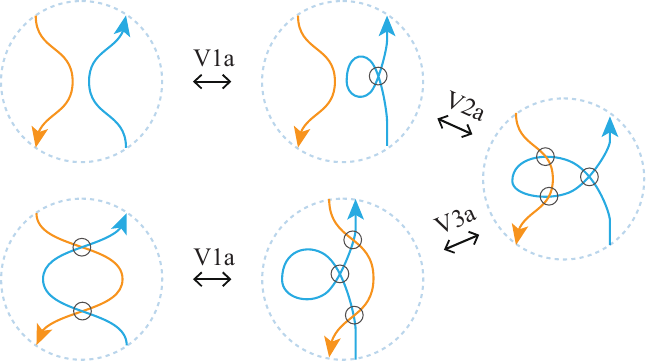}}$$
\end{proof}

We now extend our proofs by incorporating the move $V2b$ alongside the moves in the set $G$. From this point forward, we shall freely use any moves that have already been shown to be derivable from those in $G$ in the remainder of our proofs. The move $V1b$ can be realized by a sequence of $V1a$ and $V2b$ moves. This will complete the proof that the move in the set $S_1$ can be generated by the moves in $G$.

\begin{lemma}\label{Lemmav1b}
The move $V1b$ can be realized by a sequence of $V1a$ and $V2b$ moves.
\end{lemma}

\begin{proof} 
$$\centerline{\includegraphics[width=0.9\textwidth]{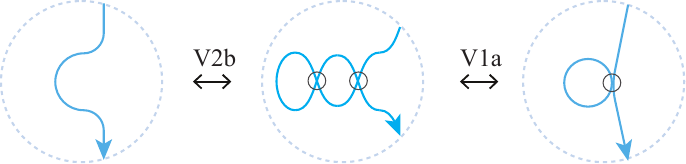}}$$
\end{proof}

The move $V2c$ can be realized by a sequence of $V1b, V2a$, and $V3a$ moves.  This will complete the proof that the moves in the set $S_2$ can be generated by the moves in $G$.

\begin{lemma}\label{Lemma1new}
The move $V2c$ can be realized by a sequence of $V1b, V2a$ and $V3a$ moves.
\end{lemma}

\begin{proof} 
$$\centerline{\includegraphics[width=0.8\textwidth]{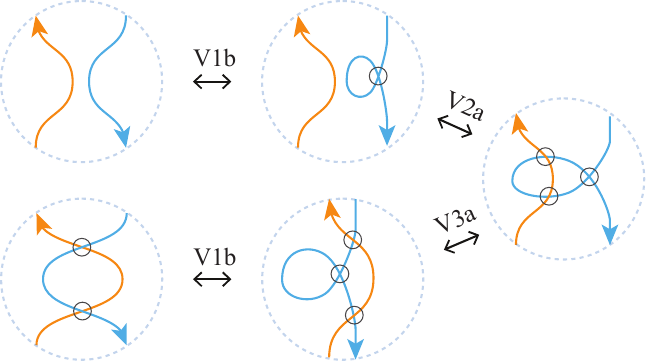}}$$
\end{proof}

Now we can use the moves in $G \cup S_1 \cup S_2$ to prove that the moves in the set $S_3$ can be obtained by a finite sequence of the moves in $G \cup S_1 \cup S_2$.

\begin{theorem}\label{theorem2}
To generate all type $V3$ ($V3b, V3c$ and $V3d$) moves, only one oriented virtual Reidemeister move of type $V3$ ($V3a$) is required.
\end{theorem}

The proof of Theorem \ref{theorem2} is a consequence of the following lemma.

\begin{lemma}\label{Lemma2}
The moves $V3b, V3c$ and $V3d$ can be realized by a sequence of moves $V2b$, $V2c$ and $V3a$.
\end{lemma}

\begin{proof} 
$$\centerline{\includegraphics[width=1\textwidth]{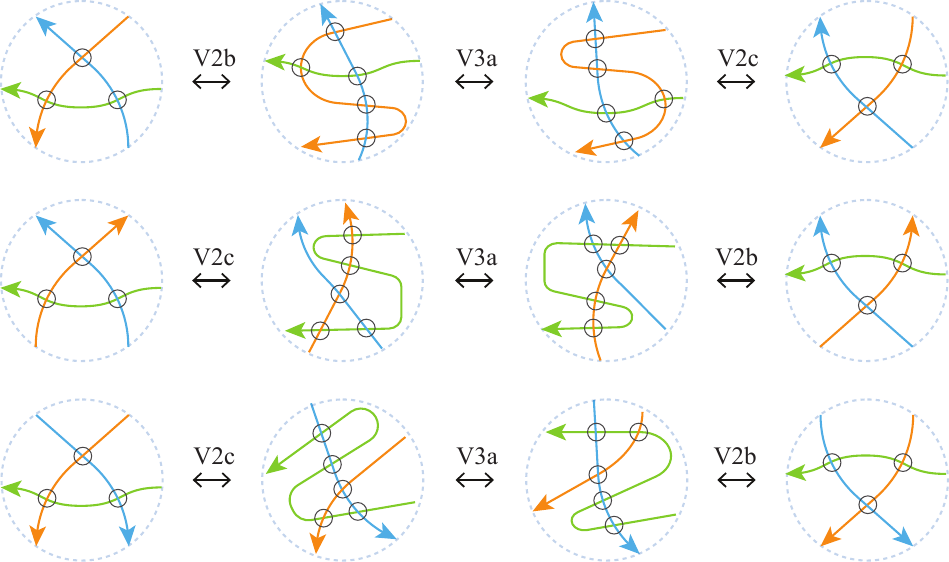}}$$
\end{proof}

At this point, all three types of oriented moves for classical ($C1-C3$) and virtual ($V1-V3$) Reidemeister moves are available, which simplifies the proof of type $V4$ moves.

\begin{lemma}\label{Lemma3}
The move $V4a$ can be realized by a sequence of moves $V2b, V2c$ and $V4g$.
\end{lemma}

\begin{proof} 
$$\centerline{\includegraphics[width=1\textwidth]{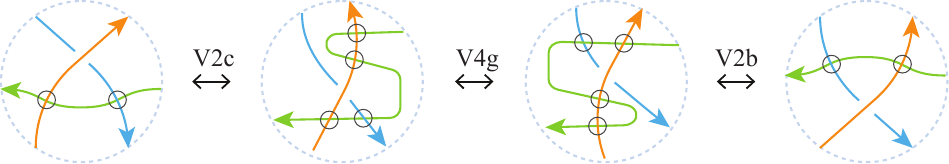}}$$
\end{proof}

\begin{lemma}\label{Lemma4}
The move $V4b$ can be realized by a sequence of moves $V2a$ and $V4g$.
\end{lemma}

\begin{proof} 
$$\centerline{\includegraphics[width=1\textwidth]{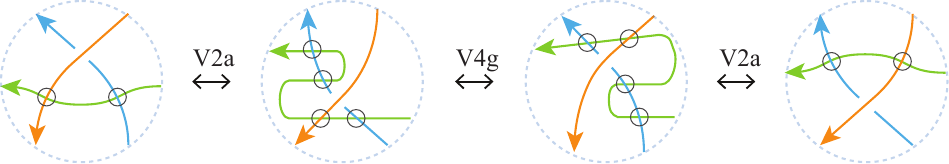}}$$
\end{proof}

\begin{lemma}\label{Lemma5}
The move $V4c$ can be realized by a sequence of moves $C2a, C2b$ and $V4g$.
\end{lemma}
\begin{proof} 
$$\centerline{\includegraphics[width=1\textwidth]{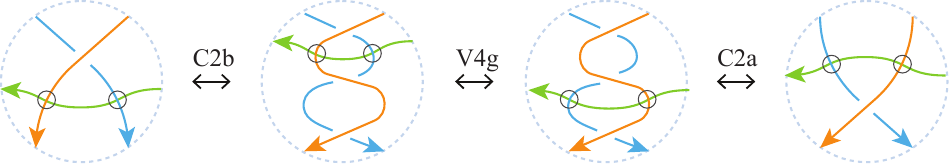}}$$
\end{proof}

\begin{lemma}\label{Lemma6}
The move $V4d$ can be realized by a sequence of moves $V2a$ and $V4c$.
\end{lemma}

\begin{proof} 
$$\centerline{\includegraphics[width=1\textwidth]{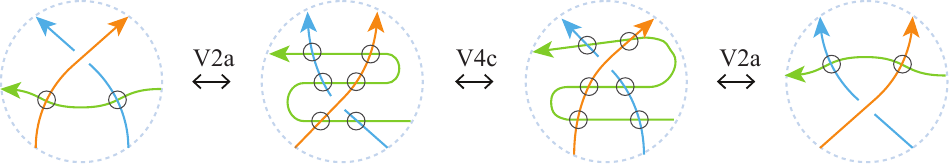}}$$
\end{proof}

\begin{lemma}\label{Lemma7}
The move $V4e$ can be realized by a sequence of moves $C2c, C2d$ and $V4b$.
\end{lemma}

\begin{proof} 
$$\centerline{\includegraphics[width=1\textwidth]{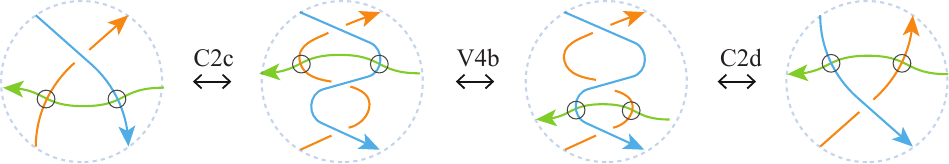}}$$
\end{proof}

\begin{lemma}\label{Lemma8}
The move $V4f$ can be realized by a sequence of moves $V2a$ and $V4d$.
\end{lemma}

\begin{proof} 
$$\centerline{\includegraphics[width=1\textwidth]{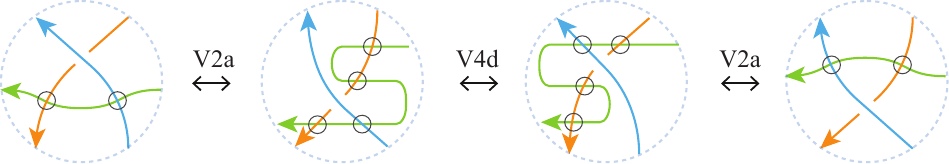}}$$
\end{proof}

\begin{lemma}\label{Lemma9}
The move $V4h$ can be realized by a sequence of moves $V2a$ and $V4g$.
\end{lemma}

\begin{proof} 
$$\centerline{\includegraphics[width=1\textwidth]{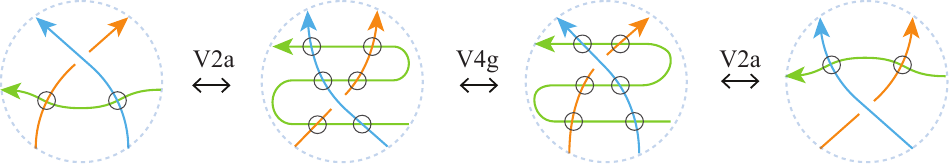}}$$
\end{proof}

This concludes the proof of Theorem \ref{theorem}.

The proof establishing a generating set for oriented virtual knot diagrams opens the door to several intriguing questions for future exploration. These include:

\begin{enumerate}
    \item Caprau and Scott in \cite{Caprau} proved that in addition to \{$C1a, C1b, C2a$, $C3a$\} there are eleven more minimal 4-element generating sets of oriented Reidemeister moves. They showed that these twelve sets encompass all possible minimal generating sets for oriented classical Reidemeister moves. Similar research can be conducted for oriented virtual Reidemeister moves, identifying alternative generating sets of oriented virtual Reidemeister moves and proving their minimality.
    
    \item In classical knot theory, the axiomatization of the Reidemeister moves leads naturally to the definition of a quandle. In \cite{gsosk}, authors generalize the structure of quandles by introducing a new algebraic structure tailored for singular oriented knots and links. The axioms of oriented singquandles are derived from a generating set of Reidemeister moves. A similar approach can be extended to virtual knot theory. This would involve developing axioms for a virtual quandle based on the generating set of oriented virtual Reidemeister moves, creating new avenues for exploring invariants and other properties in the context of virtual knot theory.
\end{enumerate}

\bibliographystyle{amsplain}

\end{document}